\documentclass[a4paper,12pt]{article}

\usepackage[textwidth=125mm, textheight=195mm]{geometry}
\usepackage[english]{babel}
\usepackage{graphicx}
\usepackage{amsmath}
\usepackage{amsfonts}
\usepackage{amsthm}
\usepackage{epstopdf}
\usepackage{color}

\begin{document}

\newcommand{\wk}{\mbox{$\,<$\hspace{-5pt}\footnotesize )$\,$}}

\numberwithin{equation}{section}
\newtheorem{teo}{Theorem}
\newtheorem{lemma}{Lemma}

\newtheorem{coro}{Corollary}
\newtheorem{prop}{Proposition}
\theoremstyle{definition}
\newtheorem{definition}{Definition}
\theoremstyle{remark}
\newtheorem{remark}{Remark}

\newtheorem{scho}{Scholium}
\newtheorem{open}{Question}
\newtheorem{example}{Example}
\numberwithin{example}{section}
\numberwithin{lemma}{section}
\numberwithin{prop}{section}
\numberwithin{teo}{section}
\numberwithin{definition}{section}
\numberwithin{coro}{section}
\numberwithin{figure}{section}
\numberwithin{remark}{section}
\numberwithin{scho}{section}

\bibliographystyle{abbrv}

\title{Some topics in differential geometry of normed spaces}
\date{}

\author{Vitor Balestro\footnote{Corresponding author}  \\ CEFET/RJ Campus Nova Friburgo \\ 28635000 Nova Friburgo \\ Brazil \\ vitorbalestro@id.uff.br \and Horst Martini \\ Fakult\"{a}t f\"{u}r Mathematik \\ Technische Universit\"{a}t Chemnitz \\ 09107 Chemnitz\\ Germany \\ martini@mathematik.tu-chemnitz.de \and  Ralph Teixeira \\ Instituto de Matem\'{a}tica e Estat\'{i}stica  \\ Universidade Federal Fluminense \\24210201 Niter\'{o}i\\ Brazil \\ ralph@mat.uff.br}

\maketitle

\begin{abstract} For a surface immersed in a three-dimensional space endowed with a norm instead of an inner product, one can define analogous concepts of curvature and metric. With these concepts in mind, various questions immediately appear. The aim of this paper is to propose and answer some of those questions. In this framework we prove several characterizations of minimal surfaces, and analogues of some global theorems (e.g., Hadamard-type theorems) are also derived. A result on the curvature of surfaces of constant Minkowski width is also given. Finally, we study the ambient metric induced on the surface, proving an extension of the classical Bonnet theorem.
\end{abstract}


\bigskip


\section{Introduction}

This is the last of a series of three papers devoted to study the differential geometry of surfaces immersed in three-dimensional real vector spaces endowed with a norm, which we call \emph{normed} (=\emph{Minkowski}) \emph{spaces}. In \cite{diffgeom}, the first paper of the series, the core of the theory was developed. There were introduced concepts of \emph{Minkowski Gaussian, mean} and \emph{principal curvatures} from regarding the normal map based on \emph{Birkhoff orthogonality}. The second paper \cite{diffgeom2} was devoted to explore the theory from the viewpoint of affine differential geometry. The aim of this third paper is to use the machinery developed previously to investigate some classical topics in our new framework. Now we briefly recall some definitions given in \cite{diffgeom} and \cite{diffgeom2}.\\

We will always work with a surface immersion $f:M\rightarrow(\mathbb{R}^3,||\cdot||)$, where the norm $||\cdot||$ is \emph{admissible}, meaning that its \emph{unit sphere} $\partial B:=\{x\in\mathbb{R}^3:||x|| = 1\}$ has strictly positive Gaussian curvature in the usual Euclidean geometry of $\mathbb{R}^3$ (also, we denote the usual inner product in this space by $\langle\cdot,\cdot\rangle$). The norm $||\cdot||$ induces an orthogonality relation between directions and planes given as follows: we say that a non-zero vector $v \in \mathbb{R}^3$ is \emph{Birkhoff orthogonal} to a plane $P$ (denoted $v \dashv_B P$) if $||v + tw|| \geq ||v||$ for any $t \in \mathbb{R}$ and $w \in P$; see \cite{alonso}. In other words, $v$ is Birkhoff orthogonal to $P$ if $P$ supports the \emph{unit ball} $B:=\{x \in \mathbb{R}^3:||x|| \leq 1\}$ at $v/||v||$. It follows from the admissibility of the norm that this relation is unique both at left (in the sense of directions) and at right (in the sense of planes). \\

For a surface immersion $f:M\rightarrow(\mathbb{R}^3,||\cdot||)$, we define the \emph{Birkhoff-Gauss map} $\eta:M\rightarrow\partial B$ as follows: we associate each $p \in M$ to a unit vector $\eta(p)$ such that $\eta(p) \dashv_B T_pM$. Notice that we have two possible choices for each point, and therefore such a map should be only locally defined. However, if the surface is orientable, then the Birkhoff-Gauss map can be defined globally, and hence we will assume this hypothesis throughout the text. As it is proved in \cite{diffgeom}, this defines an \emph{equiaffine transversal vector field} in $M$ (in the sense of \cite{nomizu}), and the associated \emph{Gauss formula} reads
\begin{align*} D_XY = \nabla_XY + h(X,Y)\eta,
\end{align*}
for any vector fields $X,Y \in C^{\infty}(TM)$, with $D$ denoting the standard connection on $\mathbb{R}^3$. The bilinear map $h$ is called the \emph{affine fundamental form}, and in some sense it plays the role of the classical \emph{second fundamental form}. Let $\xi$ denote the Euclidean Gauss map of $M$, and let $u^{-1}$ denote the Euclidean Gauss map of the unit sphere $\partial B$. Up to re-orientation, we clearly have $\eta = u \circ \xi$. Also, the following expression to $h$ is straightforward:
\begin{align}\label{exph} h(X,Y) = \frac{\langle D_XY,\xi\rangle}{\langle\eta,\xi\rangle} = -\frac{\langle Y,d\xi_pX\rangle}{\langle\eta,\xi\rangle} = -\frac{\langle du^{-1}_{\eta(p)}T,d\eta_pX\rangle}{\langle\eta,\xi\rangle},
\end{align}
for any $p \in M$ and $X,Y \in T_pM$. \\

The \emph{Minkowski Gaussian curvature} and the \emph{Minkowski mean curvature} of $M$ at $p$ are defined as $K:=\mathrm{det}(d\eta_p)$ and $H:=\frac{1}{2}\mathrm{tr}(d\eta_p)$, respectively. The \emph{principal curvatures} are the (real) eigenvalues of $d\eta_p$, and their existence is proved in \cite{diffgeom}. The associated eigenvectors are called \emph{principal directions}. The \emph{normal curvature} of $M$ at $p$ in a given direction $V \in T_pM$ was defined in \cite{diffgeom} in terms of planar sections, and can be equivalently defined as
\begin{align*} k_{M,p}(V) := \frac{\langle du^{-1}_{\eta(p)}V,d\eta_pV\rangle}{\langle du^{-1}_{\eta(p)}V,V\rangle}.
\end{align*}
As it is discussed in \cite{diffgeom2}, the normal curvature is closely associated to a Riemannian metric on $M$ called \emph{Dupin metric}. This is the metric whose unit circle, at each $p \in M$, is the (usual) Dupin indicatrix of $T_{\eta(p)}\partial B$. This is simply given by
\begin{align*} \langle X,Y\rangle_p := \langle du^{-1}_{\eta(p)}X,Y\rangle,
\end{align*}
for any $p \in M$ and $X,Y \in T_pM$. Dividing this metric by $\langle\eta(p),\xi(p)\rangle$, we obtain the \emph{weighted Dupin metric}, which will be important for our purposes. \\

We recall that all of the concepts above were defined and studied in the papers \cite{diffgeom} and \cite{diffgeom2}, and the reader is invited to consult them for a perfect acquitance with the area subject. We head now to briefly describe the structure of the present paper. In Section \ref{minimal} we study \emph{minimal surfaces} in our context, proving that we can re-obtain several characterizations of such surfaces, all of them being analogues of results in the Euclidean subcase. In Section \ref{global} we obtain some global theorems, such as \emph{Hadamard-type theorems}, as immediate consequences of their Euclidean versions. In Section \ref{constantwidth} we prove a result concerning the curvatures of \emph{constant Minkowski width surfaces}, which is also an extension of a known result of classical differential geometry. Finally, Sections \ref{metric} and \ref{control} are devoted to understand the behavior of the ambient induced metric on $M$. In particular, a version of \emph{Bonnet's classical theorem} is obtained, and we also give an estimate for the \emph{perimeter} of the normed space (in the sense of Sch\"{a}ffer, see \cite{schaffer}). \\

For general references in Minkowski geometry, we refer the reader to \cite{martini2}, \cite{martini1}, and \cite{thompson}. The differential geometry of curves in normed planes was studied in \cite{Ba-Ma-Sho}. Other approaches to the differential geometry of normed spaces can be found in \cite{Bus3}, \cite{Gug2} and \cite{cabezas}. Immersed surfaces with the induced ambient norm are, in particular, Finsler manifolds, and in this regard we refer the reader to \cite{Bus2} and \cite{Bus7}.  

\section{Minimal surfaces}\label{minimal}

Let $f:M\rightarrow(\mathbb{R}^3,||\cdot||)$ be a surface immersed in an admissible Minkowski space. We say that $M$ is a \emph{minimal surface} if its Minkowski mean curvature vanishes everywhere. In the Euclidean subcase, minimal surfaces are characterized in terms of critical points of the area functions of their normal variations. There is also an analogous result when the considered transversal vector field is the affine normal field (see \cite[Chapter III, Section 11]{nomizu}). We will see that the general Minkowski case has a similar behavior when one endows the surface with the \emph{induced area element} $\omega$ being the $2$-form defined on the tangent bundle $TM$ as
\begin{align*} \omega(X,Y) := \mathrm{det}[X,Y,\eta],
\end{align*}
where $\mathrm{det}$ is the usual determinant in $\mathbb{R}^3$. This $2$-form yields the standard area element if the considered norm is Euclidean. Hence we may define the \emph{area of} $M$ as
\begin{align*} A(M) := \int_M\omega.
\end{align*}
Let now $D \subseteq M$ be a \emph{domain} in $M$, which is an open, connected subset whose boundary is homeomorphic to a sphere. Assume that $\bar{D}\subseteq M$, where $\bar{D}$ is the union of $D$ with its boundary. Let $g:\bar{D}\rightarrow\mathbb{R}$ be any smooth function. The \emph{Birkhoff normal variation} of $\bar{D}$ with respect to $g$ is the map $F:\bar{D}\times(-\varepsilon,\varepsilon)\rightarrow(\mathbb{R}^3,||\cdot||)$ given by
\begin{align*} F(p,t) = F_t(p) = p + tg(p)\eta(p),
\end{align*}
where we identify $M$ within $\mathbb{R}^3$ with its image under $f$, as usual. It is clear that this construction yields a family of immersed surfaces parametrized by $t$. We will denote each  of these surfaces by $D_t$. Their respective Birkhoff normal vector fields and associated area elements will be denoted by $\eta_t$ and $\omega_t$. The function which associates each $t$ to the area of the surface $\bar{D}_t$ is then given by
\begin{align*} A(t) := \int_{\bar{D}}\omega_t.
\end{align*}

\begin{teo} Let $f:M\rightarrow(\mathbb{R}^3,||\cdot||)$ be an immersed surface whose Minkowski Gaussian curvature is negative. Then $M$ is a minimal surface if and only if for each domain $D \subseteq M$ and each Birkhoff normal variation of $D$ we have $A'(0) = 0$.
\end{teo}
\begin{proof} Assume that $(x,y)$ are coordinates in $D$ such that their coordinate vector fields $X:=\frac{\partial}{\partial x}$ and $Y:=\frac{\partial}{\partial y}$ are principal directions of $M$ in each point (this is possible since the Minkowski principal curvatures are different at each point). For each $p \in D$ and $t \in (-\varepsilon,\varepsilon)$, the vectors $X^t:=(F_{t})_{*}(X)$ and $Y^t:=(F_{t})_*(Y)$ span the tangent space $T_pF_t(D)$. If $\lambda_1$ and $\lambda_2$ denote the principal curvatures of $M$ at $p$, then we have
\begin{align*}X^t = (1+tg\lambda_1)X + tX(g)\eta \ \ \mathrm{and}\\
Y^t = (1+tg\lambda_2)Y + tY(g)\eta,
\end{align*}
for each $(p,t)\in D\times(-\varepsilon,\varepsilon)$. Therefore, the area function $A(t)$ writes
\begin{align*} A(t) = \int_D\omega_t(X^t,Y^t) \ dxdy = \int_D\mathrm{det}[X^t,Y^t,\eta_t] \ dxdy.
\end{align*}
Now we calculate
\begin{align*} \omega_t(X^t,Y^t) = \mathrm{det}[X^t,Y^t,\eta_t] = (1+tg\lambda_1)(1+tg\lambda_2)\mathrm{det}[X,Y,\eta_t] + tX(g)(1+tg\lambda_2)\mathrm{det}[\eta,Y,\eta_t] + \\ + tY(g)(1+tg\lambda_1)\mathrm{det}[\eta,X,\eta_t],
\end{align*}
where we assume that the basis $\{X,Y,\eta\}$ is positively oriented. For each fixed $p \in D$, the vector field $t \mapsto \eta_t(p)$ describes a curve on $\partial B$. Therefore
\begin{align*} \left.\frac{\partial}{\partial t}\eta_t(p)\right|_{t=0} \in \mathrm{span}\{X(p),Y(p)\} = T_pM\,,
\end{align*}
and hence
\begin{align*}\left. \frac{\partial}{\partial t}\omega_t(X^t,Y^t)\right|_{t=0} = g(\lambda_1+\lambda_2)\mathrm{det}[X,Y,\eta].
\end{align*}
It follows immediately that
\begin{align*} A'(0) = \int_Dg(\lambda_1+\lambda_2)\omega = \int_D2gH\omega,
\end{align*}
where $H$ denotes the Minkowski mean curvature of $M$. If $H = 0$, then we have clearly $A'(0) = 0$ for any domain $D \subseteq M$ and any Birkhoff normal variation of $D$. The converse follows from standard analysis arguments.

\end{proof}

As in the Euclidean subcase, we can characterize minimal surfaces (at least the ones of negative Minkowski Gaussian curvature) by means of the affine fundamental form (which, as we remember, plays the role of the second fundamental form). This is our next statement.

\begin{prop} Let $f:M\rightarrow(\mathbb{R}^3,||\cdot||)$ be an immersed surface with Birkhoff-Gauss map $\eta$ and affine fundamental form $h$. Assume that $M$ has negative Minkowski Gaussian curvature $K$. Then $M$ is minimal (in the Minkowski sense) if and only if there exists a function $c:M\rightarrow\mathbb{R}$ such that
\begin{align}\label{hmin} h(d\eta_pX,d\eta_pY) = c(p)\cdot h(X,Y),
\end{align}
for any $p \in M$ and $X,Y \in T_pM$. In this case, $c(p) = -K(p)$ for each $p \in M$.  
\end{prop}
\begin{proof} Let $p \in M$. Since $K(p) < 0$, we have that the principal curvatures $\lambda_1,\lambda_2 \in \mathbb{R}$ are different, and then we have associated principal directions $V_1,V_2 \in T_pM$ such that $h(V_1,V_2) = 0$. If $X,Y \in T_pM$, then we can decompose them as
\begin{align*} X = \alpha_1V_1 + \alpha_2V_2  \ \mathrm{and} \\ Y = \beta_1V_1 + \beta_2V_2.
\end{align*}
Therefore, rescaling $V_1$ and $V_2$ in order to have $h(V_1,V_1) = -h(V_2,V_2) = 1$, we have
\begin{align*} h(d\eta_pX,d\eta_pY) = h(\alpha_1\lambda_1V_1+\alpha_2\lambda_2V_2,\beta_1\lambda_1V_1+\beta_2\lambda_2V_2) = \alpha_1\beta_1\lambda_1^2 - \alpha_2\beta_2\lambda_2^2.
\end{align*}
On the other hand, $h(X,Y) = \alpha_1\beta_1 - \alpha_2\beta_2$. Hence we have (\ref{hmin}) for all $X,Y \in T_pM$ if and only if $\lambda_1^2 = \lambda_2^2$. This happens if and only if $\lambda_1 = -\lambda_2$, since the Minkowski Gaussian curvature $K = \lambda_1\lambda_2$ is negative.

\end{proof}

Recall that the \emph{weighted Dupin metric} of an immersion $f:M\rightarrow(\mathbb{R}^3,||\cdot||)$ is the metric given by
\begin{align*}  b(X,Y) := \frac{\langle du^{-1}_{\eta(p)}X,Y\rangle}{\langle\eta(p),\xi(p)\rangle},
\end{align*}
for each $p \in M$ and $X,Y \in T_pM$. It is an important fact in classical differential geometry that minimal surfaces can be characterized by their Gauss maps being conformal. Next, we prove something similar for Minkowski minimal surfaces, replacing the usual metric by the weighted Dupin metric.

\begin{teo} An immersed surface with negative Minkowski Gaussian curvature is a minimal surface if and only if its Birkhoff-Gauss map is conformal with respect to the weighted Dupin metric (and, clearly, also with respect to the Dupin metric).
\end{teo}
\begin{proof} First, notice that
\begin{align}\label{hb} h(X,Y) = -\frac{\langle du^{-1}_{\eta(p)}Y,d\eta_pX\rangle}{\langle\eta,\xi\rangle} = -b(Y,d\eta_pX).
\end{align}
Then, due to the symmetry of $h$, it follows that $d\eta_p$ is self-adjoint with respect to the weighted Dupin metric for each $p\in M$. Using the equality above and the previous proposition, we get that the equality
\begin{align*} -b(Y,d\eta_pX) = h(X,Y) = -K(p)\cdot h(d\eta_pX,d\eta_pY) = K(p) \cdot b(d\eta_pY,d\eta_p\circ d\eta_pX)
\end{align*}
holds if and only if $M$ is minimal. Setting $Z = d\eta_pX$, the above becomes 
\begin{align*} -b(Y,Z) = K(p)\cdot b(d\eta_pY,d\eta_pZ).
\end{align*}
Since $K <0$, we see that $d\eta_p$ is an isomorphism for each $p \in M$. Hence the last equality holds for any $p \in M$ and for any $Y, Z \in T_pM$ if and only if $M$ is minimal.

\end{proof}

In classical differential geometry, minimal surfaces are also characterized as immersions for which the Laplacian of the coordinate functions vanishes. We will prove something similar here. We follow \cite[Section II.6]{nomizu} to define a concept which is analogous to that of the Laplacian for functions defined over $M$. We call it the \emph{b-Laplacian} and denote it by $\Delta_bf$. For this sake (following \cite{diffgeom2} and \cite{nomizu}, and if $\hat{\nabla}$ is the \emph{Levi-Civita connection} of the metric $b$), we define the \emph{b-Hessian} of a function $f \in C^{\infty}(M)$ to be the bilinear map
\begin{align*} \mathrm{hess}_bf:= X(Yf) - (\hat{\nabla}_XY)f,
\end{align*}
for any $X,Y \in C^{\infty}(TM)$. Still following \cite{nomizu}, since $b$ is positive definite, the $b$-Laplacian can be defined simply by taking the trace of $\mathrm{hess}_bf$ with respect to $b$. Formally,
\begin{align*} \Delta_b f(p) := \mathrm{tr}_b\left(\mathrm{hess}_bf|_p\right), \ \ p \in M.
\end{align*}

Notice that this is the \emph{Laplace-Betrami operator} for the Riemannian metric $b$ on $M$. We recall here that the trace with respect to the weighted Dupin metric $b$ is calculated by taking an orthonormal basis for $b$. In the next theorem, we show that (Minkowski) minimal surfaces can be characterized as immersions for which the $b$-Laplacian of the coordinate functions vanishes.

\begin{teo} Let $f=(f_1,f_2,f_3):M\rightarrow(\mathbb{R}^3,||\cdot||)$ be an immersed surface whose Minkowski mean curvature is denoted by $H$. Then $H(p) = 0$ if and only if $\Delta_b f_1 = \Delta_b f_2 = \Delta_b f_3 = 0$. In particular, $M$ is minimal if and only if the Laplacian of its coordinate functions vanishes at every point. 
\end{teo}
\begin{proof} Let $p \in M$, and assume that $(x,y,z)$ be coordinates in $\mathbb{R}^3$ given by $(x,y,z)\mapsto p+ xV_1 + yV_2 + z\eta(p)$, where $V_1$ and $V_2$ are distinct principal directions of $M$ at $p$, nor\-ma\-li\-zed in the weighted Dupin metric (this is a \emph{Monge form parametrization}, see \cite{izumiya}). Therefore, $(x,y,g(q))$ is the position vector of $M$ in a neighborhood of $p$, where $q \in M$ is the intersection of the line $t \mapsto p+xV_1+yV_2 + t\eta(p)$ with $M$. Equality (3.4) in \cite{diffgeom2} gives that, at $p$, we have $\mathrm{hess}_bg(X,Y) = -h(X,Y)$ for any $X,Y \in T_pM$. From this and equality (\ref{hb}), and since $p$ is a critical point of $g$, we get
\begin{align*} \Delta_b g(p) = \mathrm{tr}_b(\mathrm{hess}_bg|_p) = \mathrm{tr}_b(-h) = \mathrm{tr}(d\eta_p) = 2H(p).
\end{align*}
Since we clearly have $\Delta_b x(p) = \Delta_b y(p) = 0$, the proof is complete. Notice that we can ``choose" the coordinates in $\mathbb{R}^3$ because a zero Laplacian remains zero under an affine transformation. 

\end{proof}

\section{Global theorems}\label{global}

In this section we will prove versions of the Hadamard theorems for suitable hypotheses regarding the Minkowski Gaussian curvature. We also prove that, analogously to the Euclidean subcase, if the Minkowski Gaussian curvature of a (closed) surface vanishes in every point, then this surface must be a plane or a cylinder. As we will see, these theorems come as consequences of their Euclidean ``counterparts". Throughout this section we assume that, as usual, all the norms involved are admissible. Also, we say that an immersed surface is \emph{topologically closed} if it is closed in the topology derived from the norm fixed in the space (which is, of course, the same as the topology endowed by the Euclidean norm). Assuming that the surface is topologically closed is an independent-of-the-norm way to deal with surfaces which are \emph{geodesically complete} (or simply \emph{complete}) in Euclidean differential geometry (see \cite{manfredo} for the definition and for the proof of this implication). In such a geometry, the completeness of the surface is an essential hypothesis for the theorems we aim to extend next. The following proposition is the key for the results of this section (see \cite{diffgeom} for a proof).
\begin{prop}\label{gaussposit} Let $f:M\rightarrow(\mathbb{R}^3,||\cdot||)$ be a surface immersed in an admissible normed space. The signs of the Minkowski and Euclidean Gaussian curvatures are the same at any point of $M$. 
\end{prop}

\begin{teo} Let $f:M\rightarrow(\mathbb{R}^3,||\cdot||)$ be a simply connected immersed surface, which is topologically closed. If the Minkowski Gaussian curvature of $M$ is non-positive, then $M$ is diffeormorphic to a plane.
\end{teo}
\begin{proof} The hypothesis on $M$ being closed implies that it is complete in the Euclidean geometry. From Proposition \ref{gaussposit} it follows that the Euclidean Gaussian curvature is non-positive. Hence the result comes as a consequence of the Hadamard theorem in Euclidean geometry (see \cite[Section 5.6 B, Theorem 1]{manfredo}).

\end{proof}

\begin{teo} Let $f:M\rightarrow(\mathbb{R}^3,||\cdot||)$ be a compact, connected immersed surface. If the Minkowski Gaussian curvature of $M$ is positive, then the Birkhoff-Gauss map $\eta:M\rightarrow\partial B$ is a diffeomorphism.
\end{teo}
\begin{proof} Again it follows from Proposition \ref{gaussposit} that the Euclidean Gaussian curvature of $M$ is positive. Therefore, the Euclidean Gauss map $\xi:M\rightarrow \partial B_e$ is a diffeomorphism (see \cite[Section 5.6 B, Theorem 2]{manfredo}). Since the norm is admissible, the Minkowski unit sphere $\partial B$ is itself a compact, connected immersed surface with positive Euclidean Gaussian curvature. It follows that $u^{-1}:\partial B\rightarrow\partial B_e$ is a diffeomorphism. Hence also $\eta = u\circ\xi$ is a diffeomorphism.

\end{proof}

Recall that a \emph{cylinder} is an immersed surface $M$ such that for each point $p \in M$ there is a unique line $r(p) \subseteq M$ through $p$, and if $p \neq q$, then the lines $r(p)$ and $r(q)$ are parallel or coincident.

\begin{teo} Let $f:M\rightarrow(\mathbb{R}^3,||\cdot||)$ be a topologically closed immersed surface whose Minkowski Gaussian curvature is null. Then $M$ is a cylinder or a plane.
\end{teo}
\begin{proof} The proof of this theorem is a consequence of the observation that a principal direction, where the curvature vanishes, is a direction that \textbf{always} determines tangential covariant derivatives. Consequently, the property that a principal curvature is zero at a certain point does not depend on the considered metric. Formally, let $X \in T_pM$ be a non-zero vector such that $d\eta_pX = 0$, where $\eta:M\rightarrow\partial B$ is the Birkhoff-Gauss map of $M$, as usual. The existence of such a vector is, in an admissible Minkowski space, equivalent to saying that the Minkowski Gaussian curvature of $M$ at $p$ is null. From (\ref{exph}), it follows that $h(X,Y) = 0$ for any $Y \in T_pM$, and this means that $D_XY$ is always tangential. It follows that $d\xi_pX =0$, and hence the Euclidean Gaussian curvature of $M$ at $p$ is also null. Therefore, the general case reduces to the Euclidean version of the theorem, which is proven in \cite[Section 5.8]{manfredo}.

\end{proof}

\section{Surfaces with constant Minkowski width}\label{constantwidth}

Let $f:M\rightarrow(\mathbb{R}^3,||\cdot||)$ be a compact, strictly convex immersed surface without boundary. We say that $M$ has \emph{constant Minkowski width} if the (Minkowski) distance between any two parallel supporting hyperplanes of $M$ is the same. This section is devoted to give a result on the principal curvatures of a surface of constant Minkowski width which is similar to its Euclidean version. In what follows, we denote by $S_p:=\{x \in T_pM:||x|| = 1\}$ the unit circle of $T_pM$.

\begin{teo} Let $f:M\rightarrow(\mathbb{R}^3,||\cdot||)$ be a surface of constant Minkowski width having positive Gaussian curvature, and let $p, q \in M$ be any points with parallel tangent planes. Then
\begin{align*} \frac{1}{\max_{X\in S_p}(k_{M,p}(X))} + \frac{1}{\min_{Y\in S_q}(k_{M,q}(Y))} = c,
\end{align*}
where $c \in \mathbb{R}$ is the width of $M$. 
\end{teo}
\begin{proof} Notice first that since $M$ is strictly convex, we can define a bijective mapping $g:M\rightarrow M$ which associates each $p \in M$ to the point $g(p) \in M$ such that $p$ and $g(p)$ have parallel tangent planes. Since $g\circ g = \mathrm{Id}|_M$, it is clear that $g$ is a diffeomorphism whose differential map is always an endomorphism. \\

Let $\eta:M\rightarrow\partial B$ be the outward point Birkhoff-Gauss normal map. By definition, we have that $\eta(p) = -\eta(g(p))$ for each $p \in M$. Our next step is to prove that the segment joining $p$ and $g(p)$ lies in the direction of $\eta(p)$. To do so, for each $p \in M$ let $h(p) \in g(p)\oplus T_{g(p)}M$ be such that $p - c\eta(p) = h(p)$, and let $w(p)$ be such that $g(p) + w(p) = h(p)$. Differentiating, we have
\begin{align*} X - cD_X\eta = D_Xg + D_Xw,
\end{align*}
for any $X \in T_pM$. It follows that $D_Xw$ is tangential for each $X \in T_pM$. Therefore, denoting the Euclidean Gauss map of $M$ by $\xi$, we have
\begin{align*} 0 = X\langle w,\xi\rangle = \langle D_Xw,\xi\rangle + \langle w,D_X\xi\rangle = \langle w,D_X\xi\rangle
\end{align*}
 for each $X \in T_pM$. Since the Minkowski Gaussian curvature of $M$ is positive, it follows that the Euclidean Gaussian curvature of $M$ is also positive, and therefore $X\mapsto D_X\xi = d\xi_pX$ is an isomorphism. Then we have that $w = 0$, and we get the equality
\begin{align*} p - c\eta(p) = g(p), \ \ p \in M.
\end{align*}
Similarly, we have the equality $p + c\eta(g(p)) = g(p)$. Let $V_1,V_2 \in T_pM$ be principal directions of $M$ at $p$, associated to the principal curvatures $\lambda_1,\lambda_2 \in \mathbb{R}$, respectively. Differentiating the first equality with respect to $V_1$ and $V_2$, we have 
\begin{align}\label{cw1} \begin{split} (1-c\lambda_1)V_1 = dg_pV_1 \ \ \mathrm{and}\\ 
(1-c\lambda_2)V_2 = dg_pV_2, \end{split}
\end{align}
respectively. Differentiating the second equality with respect to a vector $X \in T_pM$ we obtain $X + cd\eta_{g(p)}\circ dg_pX = dg_pX$. Let $W_1,W_2 \in T_{g(p}M$ be the principal directions of $M$ at $g(p)$ associated to principal curvatures $\mu_1,\mu_2 \in \mathbb{R}$, respectively. Substituting $X$ by $dg^{-1}_{g(p)}W_1$ and $dg^{-1}_{g(p)}W_2$, we get
\begin{align*} dg^{-1}_{g(p)}W_1 + c\mu_1W_1 = W_1 \ \ \mathrm{and}\\
dg^{-1}_{g(p)}W_2 + c\mu_2W_2 = W_2.
\end{align*}
Applying $dg_p$ on both sides, we have
\begin{align}\label{cw2} \begin{split} W_1 = (1-c\mu_1)dg_pW_1 \ \ \mathrm{and} \\
W_2 = (1-c\mu_2)dg_pW_2.  \end{split}
\end{align}
Writing $W_1$ and $W_2$ in terms of $V_1$ and $V_2$ and using (\ref{cw1}) and (\ref{cw2}), we obtain immediately
\begin{align*} (1-c\mu_1)(1-c\lambda_1) = 1 \ \ \mathrm{or} \ \ (1-c\mu_1)(1-c\lambda_2) = 1 \ \ \mathrm{and} \\
(1-c\mu_2)(1-c\lambda_1) = 1 \ \ \mathrm{or} \ \ (1-c\mu_2)(1-c\lambda_2) = 1.
\end{align*}
Notice that in both lines we have at least one of the equalities being true, since $W_1$ and $W_2$ are non-zero vectors. Now one sees that if $\lambda_1 = \lambda_2$ or $\mu_1 = \mu_2$, then the desired comes straightforwardly (each equality implies the other). Thus, assume that $\lambda_1 > \lambda_2$ and $\mu_1 > \mu_2$. Then, if $(1-c\mu_1)(1-c\lambda_1) = 1$, we must also have $(1-c\mu_2)(1-c\lambda_2) = 1$, which is a contradiction. It follows that $(1-c\mu_1)(1-c\lambda_2) = 1$, but this equality reads
\begin{align*} \frac{1}{\mu_1} + \frac{1}{\lambda_2} = c,
\end{align*}
which is the desired relation. Observe that the argument is symmetric: we have the same equality changing $\mu_1$ and $\lambda_2$ by $\mu_2$ and $\lambda_1$, respectively. 

\end{proof}

Recall that a point $p \in M$ is said to be \emph{umbilic} if the normal curvature $k_{M,p}$ is constant for every directions of $T_pM$. For a given strictly convex surface, we say that two points with parallel tangent planes are \emph{opposite points}. As an immediate consequence of the previous theorem, we have the following corollary.

\begin{coro} Let $f:M\rightarrow(\mathbb{R}^3,||\cdot||)$ be a strictly convex surface of constant Minkowski width. If $p \in M$ is a umbilic point, then its opposite point is also umbilic. Moreover, if the global maximum value of the map $p \mapsto \mathrm{max}_{V\in S_p}\left(k_{M,p}(V)\right)$ is attained for a umbilic point, then $M$ is a Minkowski sphere. The same holds for the global minimum value of the map $p\mapsto \mathrm{min}_{V\in S_p}\left(k_{M,p}(V)\right)$. 
\end{coro}
\begin{proof} We prove only the second claim, since the first one is immediate, and the third one is analogous. Suppose that 
\begin{align*} \lambda:=\max_{p\in M}\left(\max_{V\in S_p}\Big(k_{M,p}(V)\Big)\right)
\end{align*}
is attained for a umbilic point $p \in M$. Thus, if $c$ is the width of $M$, we have $\lambda = 2/c$. Assume that there exists a point $q \in M$ which is not umbilic, and let $\lambda_1 > \lambda_2$ be its principal curvatures. Let $\bar{q}$ be the opposite point to $q$, and let $\mu_1 > \mu_2$ be its principal curvatures. Since $\lambda \geq \lambda_1$, we get
\begin{align*} c = \frac{1}{\lambda_1} + \frac{1}{\mu_2} \geq \frac{c}{2} + \frac{1}{\mu_2},
\end{align*}
and it follows that $\mu_2 \geq 2/c$. Hence $\mu_1 > 2/c = \lambda$, and this is a contradiction.

\end{proof}

\section{The induced metric}\label{metric}

We want to study how the ambient metric is inherited by a surface immersed in a Minkowski space. In classical differential geometry, this is mainly done via the classical Hopf-Rinow theorem, but in that context the arguments depend heavily on the fact that \emph{geodesics} locally minimize lengths (see \cite{manfredo}). Since this cannot be directly ``translated" into the language of normed spaces, we adopt the viewpoint presented in \cite{burago}, namely regarding the surface as a length space. The arguments in this section are somehow standard in Finsler geometry, but some proofs are made easier in our context since here we can use a topologically equivalent Euclidean structure. \\

Let $f:M\rightarrow(\mathbb{R}^3,||\cdot||)$ be a connected immersed surface, and assume that $\sigma:[a,b]\rightarrow M$ is a piecewise smooth curve on $M$. The \emph{Minkowski length} $l(\sigma)$ of $\sigma$ is naturally defined as
\begin{align*} l\left(\sigma|_{[a,b]}\right):= l(\sigma) := \int_a^b||\sigma'(t)||dt.
\end{align*}
This definition endows $M$ with a \emph{length structure} (in the sense of \cite{burago}). As usual, we define a metric in $M$ as
\begin{align*} d(p,q) := \inf_{\sigma}l(\sigma), 
\end{align*}
where $p,q \in M$ and the infimum is taken over all piecewise smooth curves $\sigma:[a,b]\rightarrow M$ connecting $p$ and $q$. It is easy to see that $d:M\times M\rightarrow\mathbb{R}$ defined this way is indeed a metric in the usual sense, and we call it the \emph{induced Minkowski metric} (or \emph{distance}). Now we will briefly explore the topology induced by this metric. Our main objective in this section is to determine whether any two points in $M$ can be joined by a piecewise smooth curve whose length equals the distance between them.

\begin{prop}\label{topM} Assume that $M$ is closed with respect to the topology induced by $\mathbb{R}^3$. Then $(M,d)$ is a complete metric space. Moreover, $(M,d)$ is locally compact. 
\end{prop}
\begin{proof} In our context, this is slightly easier than in general Finsler manifolds. The reason is that we can just compare the Minkowski metric in $M$ with an auxiliary usual Euclidean metric. Let $||\cdot||_e:=\sqrt{\langle\cdot,\cdot\rangle}$ denote the Euclidean norm. Therefore, the Euclidean length $l_e(\sigma)$ of a curve $\sigma:[a,b]\rightarrow M$ is given by
\begin{align*} l_e(\sigma) = \int_a^b||\sigma'(t)||_edt.
\end{align*}
This length structure induces a metric $d_e$ defined in the same way as $d$. Since any two norms in a finite vector space are equivalent, we may fix a constant $c > 0$ such that
\begin{align*} \frac{1}{c}||\cdot||_e \leq ||\cdot|| \leq c||\cdot||_e.
\end{align*}
Thus, the same inequality holds for the Minkowski and the Euclidean lengths on $M$. Consequently, we have that the metrics $d$ and $d_e$ are equivalent:
\begin{align*} \frac{1}{c}d_e(\cdot,\cdot) \leq d(\cdot,\cdot) \leq cd_e(\cdot,\cdot).
\end{align*}
It follows that the topology induced by $d$ is the same as the topology induced by $d_e$, and therefore we can use the known results for the Euclidean subcase. Our result follows from the fact that if $M$ is closed in the topology of $\mathbb{R}^3$, then it is \emph{geodesically complete}, and hence complete as a metric space (see \cite[Chapter 5]{manfredo} and \cite[Chapter VII]{manfredo2}). \\

The fact that $(M,d_e)$ is locally compact comes from the observation that, for each $p \in M$, the \emph{exponential map} $\exp_p:T_pM\rightarrow M$ is a diffeomorphism in a neighborhood of $0 \in T_pM$. Again we refer the reader to \cite{manfredo} for further details. 
 
\end{proof}

\begin{remark} Notice that the distance associated to the length structure induced by the Dupin metric determines on $M$ the same topology as the Euclidean and Minkowski distances. To verify this, one has analogously to bound the Dupin norm in terms of the Euclidean norm, by using the extremal values of the norm operator of $du^{-1}_q$ as $q$ varies through the (compact) unit circle $\partial B$. \\
\end{remark}

Combining Proposition \ref{topM} with Theorem 2.5.23 in \cite{burago}, we have immediately the main result of this section.

\begin{teo} Let $f:M\rightarrow(\mathbb{R}^3,||\cdot||)$ be an immersed surface which is closed in the topology induced by the ambient space. Then, for any $p,q \in M$, there exists a curve $\gamma:[a,b]\rightarrow M$ joining $p$ and $q$ such that $l(\gamma) = d(p,q)$. 
\end{teo}

From now on, such minimizing curves will be called \emph{Minkowski geodesics}, or simply \emph{geodesics}. As a matter of fact, the Minkowski geodesics are smooth curves. For a proof, we refer the reader to \cite[Chapter 5]{shen}. There, the minimizing curves (or \emph{shortest paths}) are obtained as the \emph{trajectories} of the \emph{Finsler spray} (see the mentioned reference for precise definitions). It seems to be difficult to find further ``good" characterizations of the geodesics in our context. However, we can find a family of Finsler metrics on the usual $2$-sphere $S^2$ for which we can guarantee the existence of closed geodesics (the problem on finding closed geodesics in Riemannian and Finsler manifolds is a very active topic of research, see, e.g., \cite{long}). \\

It is easy to see that the unit sphere $\partial B$ of a normed space $(\mathbb{R}^3,||\cdot||)$ has infinitely many closed geodesics (in the induced ambient norm). Namely, any geodesic connecting two antipodal points must close, since the symmetry of the norm guarantees that the antipodal curve is also a geodesic. Therefore, intuitively, if we can deform isometrically $\partial B$ to become $S^2$ with a Finsler metric $F$, say, then $(S^2,F)$ has infinitely many closed geodesics. We fomalize this idea as follows. We say that a Finsler metric $F$ on $S^2$ is \emph{of immersion type} if the following holds: there exists a smooth and strictly convex body $K\subseteq\mathbb{R}^3$ (in the sense that its Euclidean Gaussian curvature is strictly positive) and a diffeomorphism $u:\partial K\rightarrow S^2$ such that
\begin{align*} F(u(x),du_xv) = ||v||,
\end{align*}
for any $x \in \partial K$ and $v \in T_x\partial K$, where $||\cdot||$ is the norm in $\mathbb{R}^3$ inherited from $K$ by the \emph{Minkowski functional} (see \cite{thompson}). In other words, a Finsler metric $F$ on $S^2$ is of immersion type if $(S^2,F)$ is (globally) isometric to the unit sphere of some admissible norm $||\cdot||$ on $\mathbb{R}^3$. 

\begin{teo} Let $F$ be a Finsler metric on $S^2$. If $F$ is of immersion type, then $(S^2,F)$ has infinitely many closed geodesics.
\end{teo}
\begin{proof} Let $\gamma:S^1\rightarrow\partial K$ be a closed geodesic of $\partial K$ (with respect to the induced metric $||\cdot||$). By definition, we have that the diffeomorphism $u:\partial K\rightarrow S^2$ is an isometry. Hence $u\circ\gamma:S^1\rightarrow S^2$ is a closed geodesic of $(S^2,F)$. Since there are infinitely many closed geodesics in $\partial K$, we have the result.

\end{proof}

\section{Estimates for perimeter and diameter}\label{control}

This section is devoted to find bounds on the Minkowski Gaussian curvature in terms of the Euclidean Gaussian curvature, with the aim of estimating the \emph{diameter} of a surface under certain hypotheses. As a consequence, we give an upper bound for the \emph{perimeter} of $(\mathbb{R}^3,||\cdot||)$ (in the sense of Sch\"{a}ffer, see \cite{schaffer}). In what follows, $K(p)$ and $K_e(p)$ denote the Minkowski and Euclidean Gaussian curvatures of a surface $M$ in a point $p \in M$, respectively. Also, $K_{\partial B}(q)$ denotes the Euclidean Gaussian curvature of the unit sphere $\partial B$ at a point $q \in \partial B$. As usual, we define the \emph{diameter} of $M$ to be the number $\mathrm{diam}(M):=\sup_{p,q\in M}d(p,q)$, where $d$ is the Minkowski metric of $M$.

\begin{lemma}\label{boundgauss} For each $p \in M$, we have the bounds
\begin{align*} mK(p) \leq K_e(p) \leq \bar{m}K(p),
\end{align*}
where $m = \inf_{q\in\partial B}K_{\partial B}(q)$ and $\bar{m} = \sup_{q\in\partial B}K_{\partial B}(q)$. 
\end{lemma}
\begin{proof} Recall that $\xi = u^{-1}\circ\eta$. For each $p \in M$, we have
\begin{align*} K_e(p) = \mathrm{det}\left(d\xi_p\right) = \mathrm{det}\left(du^{-1}_{\eta(p)}\right)\cdot\mathrm{det}\left(d\eta_p\right) = K_{\partial B}(\eta(p))\cdot K(p).
\end{align*}
The desired bounds come straightforwardly. 

\end{proof}

\begin{remark} Notice that since we are assuming that the norm is admissible, together with the compactness of $\partial B$ it follows that $0 < m,\bar{m} < \infty$. \\
\end{remark}

Now we use this to estimate the diameter of a surface whose Minkowski Gaussian curvature is bounded from below by a positive constant. The estimate is optimal in the sense that for the Euclidean case we just re-obtain Bonnet's classical theorem (cf. \cite{manfredo}). 

\begin{teo}\label{bonnet} Let $M$ be a closed surface whose Minkowski Gaussian curvature satisfies $K \geq \varepsilon > 0$. Then the diameter of $M$ (in the induced ambient metric) has the upper bound
\begin{align*} \mathrm{diam}(M) \leq \frac{\pi}{c\sqrt{m\varepsilon}},
\end{align*}
where 
\begin{align*}c = \inf_{v\in \partial B}\frac{||v||_e}{||v||} \ \ (> 0),
\end{align*}
and $m\in\mathbb{R}$ is defined as in Lemma \ref{boundgauss}. In particular, $M$ is compact. 
\end{teo}
\begin{proof} If $K \geq \varepsilon > 0$, then we have $K_e \geq m\varepsilon > 0$. Therefore, by Bonnet's theorem from classical differential geometry, it follows that
\begin{align*} \mathrm{diam}_e(M) \leq \frac{\pi}{\sqrt{m\varepsilon}},
\end{align*}
where $\mathrm{diam}_e(M)$ denotes the diameter of $M$ in the Euclidean metric. From the definition of the constant $c$, we have $c||v|| \leq ||v||_e$ for any $v \in \mathbb{R}^3$. It follows immediately that 
\begin{align*}\mathrm{diam}(M) \leq \frac{1}{c}\cdot\mathrm{diam}_e(M) \leq \frac{\pi}{c\sqrt{m\varepsilon}},
\end{align*}
and the desired follows. 

\end{proof}

\begin{remark} The compactness of $M$ under the hypothesis of Theorem \ref{bonnet} was already proved in \cite{diffgeom}. The new result here is the bound for the diameter of the surface. \\
\end{remark}

The \emph{perimeter} of a normed space is defined to be twice the supremum of the induced Minkowski distances between antipodal points of its unit sphere. We can use Theorem \ref{bonnet} to provide an upper bound for the perimeter of a normed space which only depends on the Euclidean Gaussian curvature of $\partial B$. \\

Assume that the Euclidean auxiliary structure in $\mathbb{R}^3$ is re-scaled in such a way that the Euclidean unit sphere bounds the largest Euclidean (closed) ball contained in the Minkowski (closed) ball $B$. This way, the constant $c \in \mathbb{R}$ defined in Theorem \ref{bonnet} becomes $1$ (indeed, it is attained for the touching points of $\partial B$ and $\partial B_e$).   

\begin{teo} Let $\rho(\partial B)$ denote the perimeter of a normed space $(\mathbb{R}^3,||\cdot||)$, which is assumed to be smooth and admissible. Then we have the inequality
\begin{align*} \rho(\partial B) \leq \frac{2\pi}{\sqrt{m}},
\end{align*}
where $m = \inf_{q\in\partial B}K_{\partial B}(q)$, as usual. 
\end{teo}
\begin{proof} The Minkowski Gaussian curvature of $\partial B$ equals $1$ (cf. \cite{diffgeom}). Therefore, assuming that the auxiliary Euclidean structure is re-scaled as described above and applying Theorem \ref{bonnet}, we get
\begin{align*} \mathrm{diam}(\partial B) \leq \frac{\pi}{\sqrt{m}}.
\end{align*}
Since we obviously have $\rho(\partial B) \leq 2\mathrm{diam}(\partial B)$, the result follows. 

\end{proof}

\bibliography{bibliography}

\end{document}